\documentclass[12pt]{amsart}
\usepackage{graphicx} 
\usepackage{fancyhdr}

\title{Flexibility of measurable and topological nilfactors in dynamical systems}
\author{Seljon Akhmedli}
\thanks{The author was partially supported by NSF grant DMS-2136217}

\address{Department of Mathematics, Northwestern University}
\email{seljonakhmedli2026@u.northwestern.edu}

\title[Flexibility of nilfactors in dynamical systems]{Flexibility of measurable and topological nilfactors in dynamical systems}

\usepackage{amsmath, amssymb, mathrsfs}
\usepackage{mathpazo}
\usepackage{latexsym,array,delarray,amsthm,amssymb,epsfig,setspace,tikz,amsmath,enumerate,mathrsfs,graphicx, mathtools,cite}

\usepackage{tabularx}
\usepackage{tikz-cd}
\usepackage{tikz}
\usepackage{floatrow}
\usepackage{hyperref}
\usetikzlibrary{positioning}
\usepackage{comment}
\usepackage{amsthm}
\usepackage{arydshln}

\usepackage{geometry}

\renewcommand{\labelenumi}{\arabic{enumi}.}
\renewcommand{\labelenumii}{\alph{enumii}.}

\theoremstyle{definition}

\makeindex
\def\lbr{\left\{}
\def\rbr{\right\}}

\def\msg{{\mathscr G}}
\def\msk{{\mathscr K}}
\def\msb{{\mathscr B}}
\def\msd{{\mathscr D}}
\def\mso{{\mathscr O}}
\def\msa{{\mathscr A}}
\def\mss{{\mathscr S}}
\def\msf{{\mathscr F}}
\def\msh{{\mathscr H}}
\def\msl{{\mathscr L}}
\def\msp{{\mathscr P}}
\def\msc{{\mathscr C}}
\def\mst{{\mathscr T}}
\def\msm{{\mathscr M}}
\def\msi{{\mathscr I}}
\def\msn{{\mathscr N}}
\def\cE{{\mathcal E}}
\def\cF{{\mathcal F}}
\def\cB{{\mathcal B}}
\def\cC{{\mathcal C}}
\def\cH{{\mathcal H}}
\def\cP{{\mathcal P}}
\def\cN{{\mathcal N}}
\def\cR{{\mathcal R}}
\def\cI{{\mathcal I}}
\def\cG{{\mathcal G}}
\def\cA{{\mathcal A}}
\def\mh{{\mathbb  H}}
\def\mq{{\mathbb  Q}}
\def\me{{\mathbb  E}}
\def\mb{{\mathbb  B}}
\def\mr{{\mathbb  R}}
\def\mm{{\mathbb  M}}
\def\mn{{\mathbb  N}}
\def\mz{{\mathbb  Z}}
\def\ms{{\mathbb  S}}
\def\mp{{\mathbb  P}}
\def\mi{{\mathbb I}}

\def\lnfrac#1#2{\raise.7ex \hbox{\Small $#1$}
  \kern-.15em/\kern-.15em  \lower.2ex \hbox{\Small $#2$}}

\theoremstyle{plain}
\newtheorem{theorem}{Theorem}[section]
\newtheorem{corollary}[theorem]{Corollary}
\newtheorem{lemma}[theorem]{Lemma}
\newtheorem{proposition}[theorem]{Proposition}
\newtheorem{definition}[theorem]{Definition}

\theoremstyle{definition}

\newtheorem{remark}[theorem]{Remark}

\def\thechsec{{\thechapter}.{\thesection}}
\numberwithin{equation}{section}

\usepackage{setspace}

\begin{document}

\begin{abstract}
We construct examples of minimal and uniquely ergodic systems realizing all possible behaviors in the interplay of measurable and topological nilfactors. To build such examples, we adapt an idea that stems from Furstenberg's construction of a minimal but not uniquely ergodic system on $\mathbb{T}^2$.
\end{abstract}

\newcommand{\dash}{\tikz[baseline=-0.7ex]{\draw[dashed] (0,0) -- (1.3em,0);}}

\renewcommand{\labelenumi}{\arabic{enumi}.}
\renewcommand{\labelenumii}{\alph{enumii}.}

\def\lbr{\left\{}
\def\rbr{\right\}}

\def\msg{{\mathscr G}}
\def\msk{{\mathscr K}}
\def\msb{{\mathscr B}}
\def\msd{{\mathscr D}}
\def\mso{{\mathscr O}}
\def\msa{{\mathscr A}}
\def\mss{{\mathscr S}}
\def\msf{{\mathscr F}}
\def\msh{{\mathscr H}}
\def\msl{{\mathscr L}}
\def\msp{{\mathscr P}}
\def\msc{{\mathscr C}}
\def\mst{{\mathscr T}}
\def\msm{{\mathscr M}}
\def\msi{{\mathscr I}}
\def\msn{{\mathscr N}}
\def\cE{{\mathcal E}}
\def\cF{{\mathcal F}}
\def\cB{{\mathcal B}}
\def\cC{{\mathcal C}}
\def\cH{{\mathcal H}}
\def\cP{{\mathcal P}}
\def\cN{{\mathcal N}}
\def\cR{{\mathcal R}}
\def\cI{{\mathcal I}}
\def\cG{{\mathcal G}}
\def\cA{{\mathcal A}}
\def\mh{{\mathbb  H}}
\def\mq{{\mathbb  Q}}
\def\me{{\mathbb  E}}
\def\mb{{\mathbb  B}}
\def\mr{{\mathbb  R}}
\def\mm{{\mathbb  M}}
\def\mn{{\mathbb  N}}
\def\mz{{\mathbb  Z}}
\def\ms{{\mathbb  S}}
\def\mp{{\mathbb  P}}
\def\mi{{\mathbb I}}

\pagenumbering{arabic}

\def\thechsec{{\thechapter}.{\thesection}}
\numberwithin{equation}{section}

\maketitle

\section{Introduction}

Nilsystems and nilfactors of systems arise in convergence and recurrence problems in ergodic theory and topological dynamics. Prime examples of these can be found in \cite{topological characteristic factors} and \cite{HK paper}. Numerous measure-theoretic and topological properties, such as transitivity, minimality, ergodicity, and unique ergodicity, are equivalent in nilsystems, combining results of \cite{Auslander, Leibman, Lesigne, Parry2}. These equivalences are key in many recent developments in dynamics. 

To describe the behaviors that may arise, we assume throughout that our dynamical systems are endowed with both measurable and topological structure. 
Every topological factor naturally has the structure of a measurable factor, but the converse does not hold in general. 

The measurable $k$-step nilfactor of a system has been shown in works of Host and Kra in \cite{HK paper} to control the dynamical behavior of certain multiple ergodic averages. The topological analog of these factors was first introduced by Host, Kra and Maass in \cite{Host Kra and Maass} (see Section \ref{sec: notations and preliminaries} for definitions). A more recent interplay of the measurable and topological nilfactors has appeared in \cite{a model with topological pronilfactors} where Kra, Moreira, Richter and Robertson, solved a conjecture of Erd\H{o}s on sumsets. There, it was shown how these nilfactors control certain infinite patterns in sets with positive density. 

The well-known Jewett-Krieger theorem says that every ergodic system has a strictly ergodic (minimal and uniquely ergodic) topological model \cite{{Jewett}, {Krieger}}. By a topological model, or just model, we mean a topological dynamical system equipped with an invariant probability measure that is measurably isomorphic to the original system.
In \cite{maximal pronilfactors}, it is shown that every ergodic system has a strictly ergodic model where the measurable and topological $k$-step nilfactors, $Z_k$ and $X_k$, are isomorphic. An older result of Lehrer \cite{Lehrer} obtaining topologically mixing models of ergodic systems is most relevant for this work. An important consequence of his theorem is the existence of systems where $Z_1$ and $X_1$ are not isomorphic; it then follows $Z_i$ and $X_i$ are not isomorphic as well for all $i>1$. Concrete examples of such systems can be found in, for instance, \cite{{nontrivial Kronecker trivial MEF}, {Gla}, {Lehrer}, {Parry}}. 
We can illustrate these phenomena in the diagram below
  
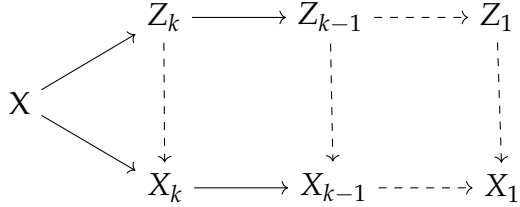
\begin{center}
\begin{tikzpicture}
  \node (X) {X};
  \node[above right=0.5cm and 1.25cm of X] (Y) {$Z_k$};
  \node[below right=0.5cm and 1.25cm of X] (Z) {$X_k$};
  \node[right=1.25cm of Y] (A) {$Z_{k-1}$}; 
  \node[right=1.25cm of Z] (B) {$X_{k-1}$};
  \node[right=1.25cm of A] (C) {$Z_1$};
  \node[right=1.25cm of B] (D) {$X_1$};

  \draw[->] (X) -- (Y);
  \draw[->] (X) -- (Z);
  \draw[->, dashed] (Y) -- (Z);
  \draw[->] (Y) -- (A);
  \draw[->] (Z) -- (B);
  \draw[->, dashed] (A) -- (B);
  \draw[->, dashed] (A) -- (C);
  \draw[->, dashed] (B) -- (D);
  \draw[->, dashed] (C) -- (D);
\end{tikzpicture}
\end{center}
 where \( Z_i \dasharrow X_i \) indicates either a measurable isomorphism for all $i$ or a proper factor map for each $i$. We construct examples realizing all possible behaviors in between these two phenomena. Understanding when the measurable and topological nilfactors are different constitutes partly to the goal of this paper. We note in a weak mixing system all nilfactors are trivial (see, for example, Proposition 20 of Chapter 8 in \cite{HK2}). We also note that nilsystems are classic examples of systems with $Z_k$ and $X_k$ isomorphic.

 \begin{theorem}\label{the theorem in introduction}
For all $0 \leq j \leq \ell \leq k,\; k\geq 1$, there exists a strictly ergodic system which admits the following:
\begin{enumerate}
    \item $Z_k$ is a nontrivial extension of $Z_{k-1}$
    \item $Z_i$ is topologically isomorphic to $X_i$ for all $0 \leq i \leq j$
    \item $Z_i$ is not measurably isomorphic to $X_i$ for all $j+1 \leq i \leq k$
    \item $X_i$ is topologically isomorphic to $X_{i+1}$ for all $i \geq \ell$.

\end{enumerate}
Furthermore, for every minimal, ergodic, and non-weak mixing $(X,\mu,T)$ there exist $j,\ell$, and $k$, such that $X$ satisfies properties (1)-(4).
 \end{theorem}

To prove Theorem \ref{the theorem in introduction}, we draw inspiration from an example of a minimal but not uniquely ergodic system constructed by Furstenberg in \cite{Fur61}. There, it is shown that for $\alpha \notin \mathbb{Q}$, the system $$T(x,y)=(x+\alpha,y+f(x)), \;\;T\colon\mathbb{T}^{2} \rightarrow \mathbb{T}^{2}$$ is minimal and not uniquely ergodic if and only if 
\begin{equation}\label{eq: coboundary furstenberg}
nf \equiv F(x+\alpha)-F(x) \mod 1
\end{equation}
has a measurable solution for some $n \in \mathbb{Z} \setminus \{0\}$ but no continuous solution unless $n=0$. In \cite{Fur61}, an $f$ which has a measurable solution for $n=1$ but no continuous solution except when $n=0$ is constructed. We refer to $f$ as a measurable but not continuous coboundary over the system $(\mathbb{T},m_{\mathbb{T}},T_{\alpha})$. In Section \ref{sec: flexibility}, we use these coboundaries to construct systems where the measurable and topological $k$-step nilfactors are not measurably isomorphic after some point.

In Proposition \ref{nontrivial measurable nils but trivial topological nils}, we exhibit another flexibility of the nilfactors where for each $k\in \mathbb{N}$, we build a strictly ergodic system which has $Z_k$ as a nontrivial extension of $Z_{k-1}$ yet has trivial topological nilfactors. Such examples are known to exist by Lehrer \cite{Lehrer}, however, we provide a construction that is more explicit using a certain minimal Cantor system built by Durand, Frank, and Maass, in \cite{nontrivial Kronecker trivial MEF} along with techniques we develop in this article from Proposition \ref{proposition with one coboundary on k}.

\section{Preliminaries}
\label{sec: notations and preliminaries}

\subsection{Definitions and notations}\label{definition}
Throughout $\mathbb{N}$ denotes the set $\{1,2, 3, \dots \}$. By a dynamical system we mean any triple of the form $(X, \mu, T)$ where $X$ is a compact metric space endowed with the Borel $\sigma$-algebra and $T$ is a homeomorphism of $X$ preserving a Borel probability measure $\mu$. Namely, for any Borel set $B \subseteq X$ we have $\mu(T^{-1}(B))=\mu(B).$ Given two compact metric spaces $X$, $Y$, and a measurable map $f\colon X \rightarrow Y$, the push forward of $\mu$ under $f$ is the invariant measure on $Y$ defined by $(f_{*}\mu)(B)=\mu(f^{-1}(B))$ for all $B \subseteq Y$.

Since we have assumed throughout that the setting of our dynamical systems has an underlying topological structure, this allows us to discuss both measurable and topological factors.

    For a system $(X,\mu,T)$, we say $(Y,\nu,S)$ is a \textit{measurable factor} of $(X,\mu,T)$ if there exists a measurable map $\pi\colon X \rightarrow Y$, \textit{the measurable factor map}, such that $(S \circ \pi)(x)= (\pi \circ T)(x)$ for $\mu$-almost every $x \in X$ and $\pi_{*} \mu=\nu.$ Furthermore, if $\pi$ is a bijection with $\pi^{-1}$ measurable, then it is a \textit{measurable isomorphism}. We say $(Y, \nu, S)$ is a \textit{topological factor} of $(X, \mu, T)$ if there exists a continuous and surjective map $\rho\colon X \rightarrow Y$, \textit{the topological factor map}, such that $(S \circ \rho)(x)= (\rho \circ T)(x)$ for all $x\in X$ and $\rho_{*}  \mu=\nu.$ If $\rho\colon X \rightarrow Y$ is a homeomorphism, then it is a \textit{topological isomorphism}. We use the notation $\overset{m}\cong$ and $\overset{t}\cong$ to denote measurable and topological isomorphisms, respectively. 

Given a $k$-step nilpotent Lie group $G$ and $\Gamma \leq G$ a discrete cocompact subgroup, $X=G/\Gamma$ is a \textit{$k$-step nilmanifold}. Now fix $g \in G$. Then $T\colon X \rightarrow X$ given by $x \mapsto  g \cdot x$ is a \textit{nilrotation.} The Haar measure on $X$ is the unique measure invariant under all nilrotations and we denote it by $m_X$. A system $(X,m_X,T)$ where $X$ is a $k$-step nilmanifold and $T$ is a nilrotation is a \textit{$k$-step nilsystem}. A special class of nilsystems we consider are \textit{affine nilsystems}. Let $A$ be a $k\times k$ unipotent integer matrix, $b\in \mathbb{T}^{k}$, and $T\colon \mathbb{T}^{k} \rightarrow \mathbb{T}^{k}$ be given by $x \mapsto Ax+b$. Then $(\mathbb{T}^{k},m_{\mathbb{T}^{k}},T)$ is a \textit{$k$-step affine nilsystem}. Note there are nilsystems which are not affine.


We say $(Y,\nu,S)$ is a \textit{maximal factor of $(X,\mu,T)$ with property $\mathcal{P}$} if every factor of $X$ with Property $\mathcal{P}$ is a factor of $Y$. For an ergodic system $(X, \mu, T)$, \textit{the measurable $k$-step nilfactor of $X$}, denoted by $Z_{k}$, is the maximal measurable factor that is an inverse limit of $k$-step nilsystems, \textit{maximal} being with respect to measurable factor maps. 
Similarly, if we have a minimal system $(X,T)$, then \textit{the topological $k$-step nilfactor of $X$}, denoted by $X_{k}$, is the maximal topological factor that is an inverse limit of $k$-step nilsystems, where \textit{maximal} is now with respect to topological factors. 
Uniqueness (up to isomorphism) for the maximal measurable nilfactors is easily seen via the semi-norm presentation given in \cite{HK paper}, and for the topological nilfactors, via the regionally proximal relation of order $k$ in \cite{Host Kra and Maass}. 
The factors $Z_0$ and $X_0$ are both the trivial one-point system. 

The factor $X_k$ is naturally endowed with Haar measure because for inverse limits of nilsystems, minimality is equivalent to unique ergodicity, the unique measure being Haar (see \cite{HK2}). In this way, we can view $X_k$ not only as a topological factor of our system but also as a measurable factor. Systems that are both minimal and uniquely ergodic are \textit{strictly ergodic}. 

Since $Z_k$ is the \textit{maximal} measurable $k$-step nilfactor, there exists a measurable factor map which we denote by $\tau_{k}\colon Z_k \rightarrow X_k$, $k\in \mathbb{N}$. Also, since $Z_k$ is in particular an inverse limit of $(k+1)$-step nilsystems, by maximality of $Z_{k+1}$, we have a measurable factor map $ Z_{k+1}\rightarrow Z_k$ and similarly we have a topological factor map $ X_{k+1} \rightarrow X_k$ for all $k\in \mathbb{N}$.



\subsection{Necessary constraints} \label{standard constraints} There are necessary constraints to the behavior of nilfactors within a system. We summarize them below:
  \begin{enumerate}

    \item If $Z_i\overset{m}\cong Z_{i+1}$ for some $i\geq 0$, then $Z_j\overset{m}\cong Z_i$ and $X_j\overset{t}\cong X_i$ for all $j \geq i$ (follows from combining Theorem 12 in Chapter 16 of \cite{HK2} and Proposition 4.11 and Theorem 10.1 of \cite{HK paper})
    
    
     
    \item If $X_i\overset{t}\cong X_{i+1}$ for some $i \geq 0$, then $X_j\overset{t}\cong X_i$ for all $j \geq i$ (see Theorem 3.8 of \cite{Dong}).
    
     \item If $Z_i\overset{m}\cong X_i$ for some $i \geq 0$, then $Z_j\overset{m}\cong X_j$ for all $j \leq i$ (since a measurable isomorphism between two systems induces bijections between their factors).

     
\end{enumerate}

\subsection{Intermediate factors} 

Let $(Y,\nu,T)$ be a system and $K$ be a compact abelian group. The system $(Y\times K,\nu \times m_K, T)$, defined by $T(y,k)=(Ty,\rho(y)+k)$ with $\rho\colon Y \rightarrow K$ measurable, is an \textit{extension of $Y$ by $K$}. For $h \in K$, the map $V_h\colon Y \times K \rightarrow Y \times K$ given by $V_h(y,k)=(y,k+h)$ is a \textit{vertical rotation}. 

The following can be deduced from the proof of Lemma 7.3 in \cite{Furstenberg and Weiss}. We use the more detailed description in \cite{HK2}.


\begin{lemma}[Host and Kra, Chapter 5, Lemma 19 \cite{HK2}] \label{intermediate extensions}
   Let $\pi\colon (X,\mu,T) \rightarrow (Y,\nu,S)$ be an extension by a compact abelian group $K$ such that $(X,\mu,T)$ is ergodic and let $p\colon W \rightarrow Y$ be an intermediate extension, meaning there exists a factor map $q\colon X \rightarrow W$ such that $p\circ q=\pi.$ Then $W$ is the quotient of $X$ under the action of some closed subgroup $H \leq K$, acting by vertical rotations.
\end{lemma}

\subsection{Coboundaries}
\begin{definition}\label{def of cob}
    For a system $(Y,\nu,T)$ and a compact abelian group $K$, we say $\rho\colon Y \rightarrow K$ is a \textbf{coboundary over $(Y,\nu,T)$} if there exists $ F\colon Y \rightarrow K$ measurable such that 
    \begin{equation} \label{eq:coboundary}
        \rho=F \circ T - F.
    \end{equation}
\end{definition}

\begin{remark} \label{cob soln are unique}
    Given an ergodic system $(X,\mu,T)$ and a coboundary $\rho\colon X \rightarrow \mathbb{T}$ over $(X,\mu,T)$, the solutions to the cohomological equation \eqref{eq:coboundary} are unique up to a constant. Indeed, suppose $\rho=f\circ T - f=g\circ T- g$ for some $f,g\colon  X \rightarrow \mathbb{T}$. We immediately have $f - g$ is $T$-invariant, hence constant by ergodicity.
\end{remark}

We refer to $\rho$ as a \textit{measurable but not continuous coboundary} if we have a measurable solution $F\colon Y \rightarrow K$ in \eqref{eq:coboundary} and there exists no continuous solution. However, $\rho$ itself is continuous throughout.

\subsection{Eigenfunctions}\label{eigenfunctions}

Given a system $(X,\mu,T)$, we say $f\colon X \rightarrow \mathbb{S}^{1}$, is an \textit{eigenfunction} if there exists $ \lambda \in \mathbb{S}^{1}$ such that $f\circ T=\lambda f$. It is a classical fact that a system having a trivial $Z_1$ (resp. $X_1$) is equivalent to being measurably (resp. topologically) weak mixing which is further equivalent to the system having no non-constant measurable (resp. continuous) eigenfunctions (see, for example, \cite{Shao and Ye} and \cite{Walters}). Furthermore, $L^2(Z_1)$ and $C(X_1)$ are spanned by the measurable and continuous eigenfunctions of $X$, respectively. In \cite{Parry}, it is mentioned for $\alpha,\beta \notin \mathbb{Q}$ rationally independent, and for a measurable but not continuous coboundary $f\colon \mathbb{T} \rightarrow \mathbb{T}$ over $(\mathbb{T},m_{\mathbb{T}},T_{\alpha})$, for the system $(\mathbb{T}^{2},m_{\mathbb{T}^{2}},T)$ where $$T(x,y)=(x+\alpha,y+f(x)+\beta),\; T\colon \mathbb{T}^{2} \rightarrow \mathbb{T}^{2},$$ $Z_1$ is $(\mathbb{T}^{2},m_{\mathbb{T}^{2}},T_{\alpha,\beta})$ and $X_1$ is $(\mathbb{T},T_{\alpha})$. Writing $f=F\circ T_{\alpha}-F$, the latter is easily seen by noticing the eigenfunctions for $T$ are $f_{n,m}(x,y)=e(nx)e(my-mF(x))$, $(n,m)\in \mathbb{Z}^{2}$, which are continuous only when $m=0$. We also note in \cite{Gla} the same example was studied but instead using the regionally proximal relation to compute $X_1$.

\section{Furstenberg coboundaries}

In this section, we adapt the measurable but not continuous coboundaries used by Furstenberg to our setting to first obtain systems which are measurably but not topologically isomorphic to certain nilsystems.

\begin{lemma} \label{coboundary is a coboundary is a ...}
     Let $k \in \mathbb{N}$, $\alpha,\beta \in \mathbb{T}\setminus\mathbb{Q}$. Then for all $0<j\leq k$ there exists a measurable but not continuous coboundary $f\colon \mathbb{T} \rightarrow \mathbb{T}$ over $(\mathbb{T},m_{\mathbb{T}},T_{\alpha})$ such that the system $(\mathbb{T}^{k},m_{\mathbb{T}^{k}},T)$ given by 
\begin{align*}
T(x_1,x_2,\dots,x_{k})=\Big(&x_1+\alpha,x_2+x_1,\dots,x_{j}+x_{j-1},x_{j+1}+f(x_1)+x_{j},\\ &x_{j+2}+x_{j+1},\dots,x_{k}+x_{k-1}\Big)
\end{align*}
     is strictly ergodic and measurably isomorphic to $(\mathbb{T}^{k},m_{\mathbb{T}^{k}},S)$ where $$S(x_1,x_2,\dots,x_{k})=(x_1+\alpha,x_2+x_1,\dots,x_{k}+x_{k-1}).$$ Furthermore, there exists a measurable but not continuous coboundary $g\colon\mathbb{T}\rightarrow \mathbb{T}$ over $(\mathbb{T},m_{\mathbb{T}},T_{\alpha})$ such that the system $(\mathbb{T}^{k+1},m_{\mathbb{T}^{k+1}},R)$ with $$R(x_1,x_2,\dots,x_{k+1})=(x_1+\alpha,x_2+g(x_1)+\beta,x_3+x_2,\dots,x_{k+1}+x_k)$$ is also strictly ergodic and measurably isomorphic to $(\mathbb{T}^{k+1},m_{\mathbb{T}^{k+1}},S')$ where $$S'(x_1,x_2,\dots,x_{k+1})=(x_1+\alpha,x_2+\beta,x_3+x_2,\dots,x_{k+1}+x_k).$$
\end{lemma}

\begin{proof}

First we observe if 
\begin{equation}\label{eq: all coboundaries}
f=G_1 \circ T_\alpha -G_1, G_1=G_2 \circ T_\alpha-G_2,\dots, G_{k-j-1}=G_{k-j} \circ T_\alpha -G_{k-j},
\end{equation}
for some measurable $G_i\colon \mathbb{T} \rightarrow \mathbb{T}, 1 \leq i \leq k-j,$ 
then the map $$\pi(x_1,x_2,\dots,x_k)=\Big(x_1,x_2,\dots,x_j,x_{j+1}-G_1(x_1),x_{j+2}-G_2(x_1),\dots,x_k-G_{k-j}(x_1)\Big)$$
gives a measurable isomorphism between $(\mathbb{T}^{k},m_{\mathbb{T}^{k}},T)$ and $(\mathbb{T}^{k},m_{\mathbb{T}^{k}},S).$ 
Thus, to prove the lemma, it is enough to find a measurable but not continuous coboundary $f$ and measurable coboundaries $G_i$ over $(\mathbb{T},m_{\mathbb{T}},T_\alpha)$. 


To guarantee $f$ is real-valued, we can arrange $\widehat{G_1}(n)=\widehat{G_1}(-n)$ where for $i\geq 1$, $\widehat{G_i}$ indicates the $i$-th Fourier coefficient of $G_i$. Since each of the $G_i$ are coboundaries, for all $1 \leq i \leq k-j$,
\begin{equation}\label{eq: all related to first coefficients}
\widehat{G}_i(n):=\frac{\widehat{G_1}(n)}{(e(n\alpha)-1)^{i-1}}.
\end{equation}
Furthermore, since all of the functions $G_i \in L^2(\mathbb{T})$, we need
\begin{equation}\label{eq:square summable}
\sum |\widehat{G_i}(n)|^{2} <\infty.
\end{equation}
To make $f$ continuous, it suffices to have 
\begin{equation}\label{eq: continuous}
    \sum |\widehat{G_1}(n) (e(n\alpha)-1)| <\infty.
\end{equation}
To ensure $G_1$ is not continuous, by Fej\'er's theorem (see Theorem 3.1(b) in \cite{Katznelson}) it is enough for the partial sum of the Fourier series of $G_1$ to not Ces\'aro uniformly converge to $G_1$, that is,
\begin{equation}\label{eq: non cesaro}
    \displaystyle \frac{1}{N} \sum_{m=-N}^{N} \sum_{n=-m}^{m} \widehat{G_1}(n) e(nx)\; \mathrm{does\; not\; converge\; to\;} G_1(x)\; \mathrm{as\;} N \rightarrow \infty.
\end{equation}
We now construct such $f$ and $G_i$ which obey the properties above.

Indeed, since $\alpha \notin \mathbb{Q}$, by minimality of $(\mathbb{T},m_{\mathbb{T}},T_{\alpha})$, for all $\varepsilon>0$ there exists a subsequence of integers $\{n_{r}\}_{r\in \mathbb{Z}}$ such that
\begin{equation}\label{eq:first inequality}
\frac{1}{|r|^{2\varepsilon}} \leq ||n_r \alpha||< \frac{1}{|r|^{\varepsilon}}
\end{equation} where $||\cdot||$ denotes the closest distance to an integer.
Define $\widehat{G}_1(n)=\frac{1}{|r|}$, when $n=n_r$ for some $r \in \mathbb{Z}$, and $0$ otherwise. Then using inequality (\ref{eq:first inequality}), condition (\ref{eq: continuous}) is satisfied, and for sufficiently small $\varepsilon>0$, where $\varepsilon$ depends on $k-j$, condition (\ref{eq:square summable}) is satisfied. 
Notice the sequence $s_m:=\sum_{n=-m}^{m}\widehat{G_1}(m)$ diverges, therefore $$\displaystyle \frac{1}{N} \sum_{m=-N}^{N} \sum_{n=-m}^{m} \widehat{G_1}(n) \rightarrow \infty$$ satisfying condition (\ref{eq: non cesaro}) at $x=0$. Hence $$f(x)=\sum_{r \neq 0} \frac{e(n_r \alpha)-1}{|r|}e(n_r x)$$ is our desired measurable but not continuous coboundary.

It is clear $T$ preserves $m_{\mathbb{T}^{k}}$. Since $\alpha$ is irrational, $S$, and therefore $T$, must be ergodic. Also, notice $(\mathbb{T}^{k},m_{\mathbb{T}^{k}},T)$ is a sequence of ergodic group extensions by $\mathbb{T}$, starting from the uniquely ergodic system $$(x_1,x_2,\dots,x_j)\mapsto(x_1+\alpha,x_2+x_1,\dots,x_j+x_{j-1}).$$ Moreover, ergodic group extensions of uniquely ergodic systems are uniquely ergodic (see Proposition 3.10 in \cite{Fur81}), implying unique ergodicity of $T$. Since $\mathrm{Supp}(m_{\mathbb{T}^{k}})=\mathbb{T}^{k}$, we have $T$ is minimal and obtain strict ergodicity (see Proposition 3.6 in \cite{Fur81}). 

The same argument can be applied to show $(\mathbb{T}^{k+1},m_{\mathbb{T}^{k+1}},R)$ is isomorphic to $(\mathbb{T}^{k+1},m_{\mathbb{T}^{k+1}},S')$ with the only difference being \eqref{eq: all coboundaries} becomes
\begin{equation}\label{eq: all coboundaries 2}
f=G_1 \circ T_\alpha -G_1, G_1=G_2 \circ T_\alpha-G_2,\dots, G_{k-1}=G_{k} \circ T_\alpha -G_{k},
\end{equation}
for some measurable $G_i\colon \mathbb{T} \rightarrow \mathbb{T}, 1 \leq i \leq k.$ \end{proof}

The following is a rephrasing of a result from \cite{Fur61} but is also an immediate consequence of Lemma \ref{coboundary is a coboundary is a ...}. 

\begin{corollary}[Furstenberg \cite{Fur61}] \label{meas but not continuous coboundaries always exists}
    For all $\alpha \in \mathbb{T} \setminus \mathbb{Q}$, there exists a measurable but not continuous $\mathbb{T}$-valued coboundary over $(\mathbb{T},m_{\mathbb{T}},T_{\alpha}).$ 
\end{corollary}

\begin{remark}\label{two coboundaries does not change much}
    Let $k\in \mathbb{N}$ and $\alpha\notin \mathbb{Q}$. By Lemma \ref{coboundary is a coboundary is a ...} there exists a measurable but not continuous coboundary $f\colon \mathbb{T} \rightarrow \mathbb{T}$ over $(\mathbb{T},m_{\mathbb{T}},T_\alpha)$ such that the system on $\mathbb{T}^{k}$ given by 
    \begin{align*}
    (x_1,x_2,\dots,x_k)\longmapsto (&x_1+\alpha,x_2+x_1,\dots,x_{\ell}+x_{\ell-1},x_{\ell+1}+f(x_1)+x_{\ell},\\ &x_{\ell+2}+x_{\ell+1},\dots,x_{k}+x_{k-1})
    \end{align*}
    is strictly ergodic and measurably isomorphic to the latter without the $f$. Since cohomologous extensions over the same system give rise to isomorphic systems, for $\beta \notin \mathbb{Q}$ rationally independent from $\alpha$, the system on $\mathbb{T}^{k+1}$ defined by
    \begin{align*}
    (x_1,x_2,\dots,x_{k+1}) \longmapsto (&x_1+\alpha,x_2+x_1,\dots,x_{\ell}+x_{\ell-1},x_{\ell+1}+x_{\ell},\\ &x_{\ell+2}+x_{\ell+1}\dots,x_{k}+x_{k-1},x_{k+1}+f(x_1)+\beta)
    \end{align*}
    is strictly ergodic and measurably isomorphic to $$(x_1,x_2,\dots,x_{k+1}) \longmapsto (x_1+\alpha,x_2+x_1,\dots,x_{k}+x_{k-1},x_{k+1}+\beta).$$ Combining these, the latter system is measurably isomorphic to \begin{align*}
      (x_1,x_2,\dots,x_{k+1})\longmapsto\Big(&x_1+\alpha,x_2+x_1,\dots,x_{\ell}+x_{\ell-1},x_{\ell+1}+f(x_1)+x_{\ell},\\ &x_{\ell+2}+x_{\ell+1},\dots,x_{k}+x_{k-1},x_{k+1}+f(x_1)+\beta\Big).
  \end{align*}
\end{remark}

\section{Flexibility}\label{sec: flexibility}

To prove Theorem \ref{the theorem in introduction}, we begin with a lemma which characterizes when the factors $Z_k$ and $X_k$ are isomorphic.
 \begin{lemma} \label{Zk and Xk measurably iso iff cont factor map}
        Let $(X,\mu,T)$ be a minimal and ergodic system. Then the following are equivalent: 
        \begin{enumerate}
    \item $Z_k$ is measurably isomorphic to $X_k$
    \item $Z_k$ is topologically isomorphic to $X_k$
    \item There exists a topological factor map from $\pi\colon X \rightarrow Z_k$
\end{enumerate}

\end{lemma}

\begin{proof}
     We show $(1) \Rightarrow (3) \Rightarrow (2)$ and note $(2) \Rightarrow (1)$ is immediate. Let $\tau_{k}\colon Z_k \rightarrow X_k$ be a measurable isomorphism and $\rho_k\colon X \rightarrow X_k$ be a topological factor map. Since a measurable factor map in an inverse limit of nilsystems agrees almost everywhere with a topological factor map (see Theorem A.1. in \cite{Host Kra and Maass}), there exists a topological factor map from $X_k \rightarrow Z_k$, which we denote by $\tau_{k} ^{-1}$. Then $\tau_{k}^{-1} \circ \rho_k\colon X \rightarrow Z_k$ is a topological factor map. For $(3) \Rightarrow (2)$, since $X_k$ is the \textit{maximal} topological $k$-step nilfactor of $X$ we obtain $Z_k$ as a topological factor of $X_k$. Also $X_k$ is topological factor of $Z_k$ (see Theorem A.1. in \cite{Host Kra and Maass}). By uniqueness of $X_k$, we have $Z_k \overset{t}\cong X_k$.
\end{proof}

Due to Lemma \ref{Zk and Xk measurably iso iff cont factor map}, we do not need to distinguish measurable and topological isomorphisms between the measurable and topological nilfactors. However, for clarity, we continue to do so.

\begin{lemma}\label{lemma with three properties}
    Let $0< j \leq  k$ and $\alpha \in \mathbb{T}\setminus\mathbb{Q}$. Then there exists a measurable but not continuous coboundary $f\colon \mathbb{T} \rightarrow \mathbb{T}$ over $(\mathbb{T},m_{\mathbb{T}},T_{\alpha})$ such that the system $(\mathbb{T}^{k},m_{\mathbb{T}^{k}},T)$ given by 
  \begin{align*}
    T(x_1,x_2,\dots,x_{k})=\Big(&x_1+\alpha,x_2+x_1,\dots,x_{j}+x_{j-1},x_{j+1}+f(x_1)+x_{j},\\
    &x_{j+2}+x_{j+1},\dots,x_{k}+x_{k-1}\Big)
  \end{align*}
 is strictly ergodic and satisfies the following properties:

  \begin{enumerate}
   \item $Z_k$ is a nontrivial extension of $Z_{k-1}$
      \item $Z_{i}$ and $X_{i}$ are topologically isomorphic, for all $0 \leq i \leq j$
      \item $Z_i$ is not measurably isomorphic to $X_i$ for all $j+1 \leq i \leq k$.
\end{enumerate}
\end{lemma}

\begin{proof}
     By Lemma \ref{coboundary is a coboundary is a ...} there exists a measurable but not continuous coboundary $f\colon \mathbb{T} \rightarrow \mathbb{T}$ over $(\mathbb{T},m_{\mathbb{T}},T_{\alpha})$ 
      such that $(\mathbb{T}^{k},m_{\mathbb{T}^{k}},T)$ is strictly ergodic and measurably isomorphic to $(\mathbb{T}^{k},m_{\mathbb{T}^{k}},S)$ where $$S(x_1,x_2,\dots,x_k)=(x_1+\alpha,x_2+x_1,\dots, x_k+x_{k-1}).$$ (Note $T$ is defined as in the statement of the lemma.) Write $f=F\circ T_{\alpha}-F$ for some $F\colon \mathbb{T} \rightarrow \mathbb{T}$ which is not continuous.

We now compute the measurable $k$-step nilfactors of $(\mathbb{T}^{k},m_{\mathbb{T}^{k}},T)$. It suffices to describe the $k$-step nilfactors of $(\mathbb{T}^{k},m_{\mathbb{T}^{k}},S)$.
For all $1\leq i \leq k$, the $Z_i$ factor for $(\mathbb{T}^{k},m_{\mathbb{T}^{k}},S)$ is given by $$(x_1,x_2,\dots,x_{i})\longmapsto (x_1+\alpha,x_2+x_1,\dots,x_{i}+x_{i-1})$$ (see Chapter 11, Proposition 1 in \cite{HK2}). Property 1 immediately follows. For all $0\leq i\leq j$, the factor map $\pi_{i}\colon X \rightarrow Z_{i}$ is continuous (it is a standard projection map from $\mathbb{T}^{k}$ to $\mathbb{T}^{i}$), hence by Lemma \ref{Zk and Xk measurably iso iff cont factor map}, $Z_i$ is topologically isomorphic to $X_i$, establishing property (2). To show property (3), it is enough to check $Z_{j+1}$ is not measurably isomorphic to $X_{j+1}$.

By Lemma \ref{Zk and Xk measurably iso iff cont factor map}, it suffices to show there does not exist a continuous factor map $\pi\colon \mathbb{T}^{k} \rightarrow Z_{j+1}$. Any such factor map is of the form $$\pi(x_1,x_2,\dots,x_{k})=\Big(\pi_1(x_1,x_2,\dots,x_{k}),\pi_2(x_1,x_2,\dots,x_{k}),\dots,\pi_{j+1}(x_1,x_2,\dots,x_{k})\Big)$$ for some $\pi_i\colon \mathbb{T}^{k}\rightarrow \mathbb{T}, 1\leq i \leq j+1$, together with the following commutativity relation: $$(\pi_1\circ T, \pi_2 \circ T,\dots,\pi_{j+1} \circ T)=(\pi_1+\alpha,\pi_2+\pi_1,\dots,\pi_{j+1}+\pi_{j}).$$ Equating coordinates we have 
\begin{align*}
    \pi_1\circ T - \pi_1&= \alpha \;\; &(1) \\
    \pi_2\circ T-\pi_2&=\pi_1\;\; &(2) \\
  \vdots \\
    \pi_{j+1} \circ T-\pi_{j+1}&=\pi_{j}  &(j+1).
\end{align*}

\noindent Notice $\pi_1(x_1,x_2,\dots,x_{k})=x_1$ is a solution to (1) and by Remark \ref{cob soln are unique} the general solution to (1) must therefore be of the form $\pi_1(x_1,x_2,\dots,x_{k})=x_1+c_1$ for some $c_1 \in \mathbb{T}$. 
Inductively, we obtain a general solution $\pi_{j+1}$ for equation $(j+1)$ and notice that all factor maps are of the form: 
\begin{align*}
\pi(x_1,x_2,\dots,x_{k})=\Big(&x_1+c_1,x_2+d_1x_1+c_2,\dots, x_j+\sum_{i=1}^{j-1} d_{j+1-i}x_i+c_j,\\ &x_{j+1}-F(x_1)+\sum_{i=1}^{j}d_{j+1-i}x_{i} + c_{j+1}\Big)
\end{align*}
for some constants $c_i,d_i \in \mathbb{T}$. These factor maps are not continuous because $F$ is not continuous. \end{proof}


The next lemma is used in the proof of Proposition \ref{proposition with one coboundary on k}.

\begin{lemma}\label{ofcourse this is true}
      Let $n \in \mathbb{Z} \setminus \{0\}$, $\alpha,\beta \in \mathbb{T}\setminus \mathbb{Q}$ be rationally independent, and $\gamma \in \mathbb{T}$. Consider the systems $(\mathbb{T}^{3},T)$ and $(\mathbb{T}^{2},S)$ where $T(x,y,z)=(x+\alpha,y+\beta,z+y)$ and $S(x,y)=(x+\alpha,y+nx+\gamma)$. Then for all $\gamma$, $(\mathbb{T}^{2},S)$ is never a factor of $(\mathbb{T}^{3},T)$.
  \end{lemma}

  \begin{proof}
       Let $\pi\colon\mathbb{T}^{3} \rightarrow \mathbb{T}^{2}$ be any factor map from $(\mathbb{T}^3,T)$ to $(\mathbb{T}^{2},S)$. Then $\pi(x,y,z)=(\pi_1(x,y,z),\pi_2(x,y,z))$ for some maps $\pi_1,\pi_2\colon \mathbb{T}^{3}\rightarrow \mathbb{T}$, together with $\pi_1 \circ T-\pi_1=\alpha$ and $\pi_2 \circ T- \pi_2=n \pi_1 +\gamma$. Using Remark \ref{cob soln are unique}, we see for any $c \in \mathbb{T}$, $\pi_1(x,y,z)=x+c$ is a solution to the first coboundary equation. Then $\pi_2 \circ T-\pi_2=nx+nc+\gamma$. However, this implies $nx+nc+\gamma$ is a coboundary over $(\mathbb{T}^{3},T)$, which further implies that $$(x,y,z,t) \mapsto (x+\alpha,y+\beta,z+y,t+nx+nc+\gamma)$$ is not ergodic, giving a contradiction.
  \end{proof}


\begin{proposition} \label{proposition with one coboundary on k}
    Let $0<j \leq  k$, $\alpha,\beta \in \mathbb{T} \setminus \mathbb{Q}$ be rationally independent. Then there exists a measurable but not continuous coboundary $f\colon\mathbb{T} \rightarrow \mathbb{T}$ over $(\mathbb{T},m_{\mathbb{T}},T_{\alpha})$ such that the system $(\mathbb{T}^{k},m_{\mathbb{T}^{k}},T)$ given by 
  \begin{align*}
    T(x_1,x_2,\dots,x_{k})=\Big(&x_1+\alpha,x_2+x_1,\dots,x_{j}+x_{j-1},x_{j+1}+f(x_1)+x_{j},\\
    &x_{j+2}+x_{j+1},\dots,x_{k}+x_{k-1}\Big)
  \end{align*}
 is strictly ergodic and satisfies the following properties:

  \begin{enumerate}
   \item $Z_k$ is a nontrivial extension of $Z_{k-1}$
      \item $Z_{i}$ and $X_{i}$ are topologically isomorphic, for all $0 \leq i \leq j$
      \item $Z_i$ is not measurably isomorphic to $X_i$ for all $j+1 \leq i \leq k$
      \item $X_i$ is topologically isomorphic to $X_{i+1}$ for all $i \geq j$
  \end{enumerate}
where $Z_k$ is the maximal measurable nilfactor of $(\mathbb{T}^{k},m_{\mathbb{T}^{k}},T)$ and $X_j$ is the maximal topological nilfactor of $(\mathbb{T}^{k},m_{\mathbb{T}^{k}},T)$. Furthermore, for $j=0$, there exists a measurable but not continuous coboundary $g\colon \mathbb{T} \rightarrow \mathbb{T}$ over $(\mathbb{T},m_{\mathbb{T}},T_{\alpha})$ such that the system $(\mathbb{T}^{k+1},m_{\mathbb{T}^{k+1}},R)$ where $$R(x_1,x_2,\dots,x_{k+1})=(x_1+\alpha,x_2+g(x_1)+\beta,x_3+x_2,\dots,x_{k+1}+x_k)$$ is strictly ergodic and has properties (1)-(4) as well.
\end{proposition}

\begin{proof}

For the case when $0<j\leq k$, $(\mathbb{T}^{k},m_{\mathbb{T}^{k}},T)$ has properties (1), (2), and (3) by Lemma \ref{lemma with three properties}. We verify property (4) later in the proof.

For the case when $j=0$, by Lemma \ref{coboundary is a coboundary is a ...}, there exists a measurable but not continuous coboundary $g\colon \mathbb{T} \rightarrow \mathbb{T}$ such that $(\mathbb{T}^{k+1},m_{\mathbb{T}^{k+1}},R)$ is measurably isomorphic to the system given by $$(x_1,x_2,\dots,x_{k+1})\mapsto (x_1+\alpha,x_2+\beta,x_3+x_2,\dots,x_{k+1}+x_k).$$ Therefore, $Z_i$ is defined by $$(x_1,x_2,\dots,x_{i+1})\mapsto (x_1+\alpha,x_2+\beta,x_3+x_2,\dots,x_{i+1}+x_i),$$ immediately verifying property (1) (see \cite{HK2}, Chapter 11, Proposition 1). Property (2) is trivial. For property (3), it suffices to find a measurable but not continuous eigenfunction of $R$. Writing $g=F\circ T_{\alpha}-F$ we see $f(x_1,x_2,\dots,x_{k+1})=e(x_1)e(x_2-F(x_1))$ is an eigenfunction of $R$ which is not continuous because $F$ is not continuous. 

To verify property (4) for both systems, it is enough to show $X_j \overset{t}\cong X_{j+1}$ when $0<j\leq k$ for $(\mathbb{T}^{k},m_{\mathbb{T}^{k}},T)$ and $X_1 \overset{t}\cong X_2$ for $(\mathbb{T}^{k+1},m_{\mathbb{T}^{k+1}},R)$.

\textit{Case 1. $0 <j \leq k$}

Recall, the factors $Z_{j+1}$ and $X_j$ were computed in the proof of properties (1) and (2) in Lemma \ref{lemma with three properties}. First, notice $Z_{j+1}$ is a group extension of $X_j$ by $\mathbb{T}$. Hence, we can apply
Lemma \ref{intermediate extensions} and deduce there exists a closed subgroup $H \leq \mathbb{T}$ such that $X_{j+1}$ is a quotient of $Z_{j+1}$ by the following closed and $T$-invariant equivalence relation
   $$(x_1,\dots,x_j,x_{j+1})\sim (x_1,\dots,x_j,x_{j+1}') \;\;\textrm{if $x_{j+1}-x_{j+1}'\in H$}.$$ Since the only closed subgroups of $\mathbb{T}$ are finite cyclic subgroups and $\mathbb{T}$ itself, the only nontrivial intermediate factors that can give rise to the factor $X_{j+1}$ are 
   $$(x_1,x_2,\dots, x_{j+1})\longmapsto (x_1+\alpha,x_2+x_1,\dots,x_{j}+x_{j-1},x_{j+1}+nx_j), \; n\in \mathbb{Z} \setminus \{0\}.$$ Continuing as in the proof of property (3) in Lemma \ref{lemma with three properties}, we see there are no continuous factor maps to these factors.
   
  \textit{ Case 2. $j=0$}

  Recall, the $Z_2$ factor of $(\mathbb{T}^{k+1},m_{\mathbb{T}^{k+1}},R)$ is given by the 2-step affine nilsystem $$(x_1,x_2,x_3) \mapsto (x_1+\alpha,x_2+\beta,x_3+x_2),$$ and that $X_2$ is always a factor of $Z_2$. A factor of an ergodic $k$-step affine nilsystem is again an affine $k$-step nilsystem (follows from the proof of Theorem 11 in Chapter 13 of \cite{HK2}), therefore, $X_2$ is a 2-step affine nilsystem. Furthermore, it can be checked every 2-step affine nilsystem is of the form $(Z \times K, S)$ where $Z$ is the $X_1$ factor of $S$ given by some fixed $a\in Z$, $K$ is a finite dimensional torus, and $S(z,k)=(z+a, k+\phi(z)+c)$ for some continuous group homomorphism $\phi\colon Z \rightarrow K$ and constant $c\in K$. The only continuous eigenfunctions of $R$ are $f_n(x_1,x_2,\dots,x_{k+1})=e(nx_1),n\in \mathbb{Z},$ implying the $X_1$ factor of $X_2$ is $(\mathbb{T},T_{\alpha})$. Therefore, $X_2$ is of the form $$(x_1,x_2,x_3)\mapsto (x_1+\alpha,(x_2,x_3)+\phi(x_1)+c)$$ for some continuous group homomorphism $\phi\colon\mathbb{T} \rightarrow \mathbb{T}^{2}$ and constant $c=(c_1,c_2)\in \mathbb{T}^2.$ 
  All such homomorphisms are of the form $\phi_{n,m}(x)=(nx,mx),$ $ (n,m) \in \mathbb{Z}^2$; hence, $X_2$ is defined by $$(x_1,x_2,x_3) \mapsto (x_1+\alpha,x_2+nx_1+c_1,x_3+mx_1+c_2).$$ Notice neither $n$ nor $m$ are 0, for otherwise we contradict the $X_1$ factor of $X_2$ being $(\mathbb{T},T_{\alpha})$. Then the system $(x_1,x_2) \mapsto (x_1+\alpha,x_2+nx_1+c_1)$ is a factor of $X_2$ which recall is further a factor of $Z_2$. Lemma \ref{ofcourse this is true} concludes the proof. \end{proof}

\begin{proposition} \label{proposition with two coboundaries on k}
    Let $1\leq \ell \leq  k$ and $\alpha,\beta \in \mathbb{T} \setminus \mathbb{Q}$ be rationally independent. Then there exists a measurable but not continuous coboundary $f\colon\mathbb{T} \rightarrow \mathbb{T}$ over $(\mathbb{T},m_{\mathbb{T}},T_{\alpha})$ such that the system $(\mathbb{T}^{k+1},m_{\mathbb{T}^{k+1}},T)$ given by 
  \begin{align*}
      T(x_1,x_2,\dots,x_{k+1})=\Big(&x_1+\alpha,x_2+x_1,\dots,x_{\ell}+x_{\ell-1},x_{\ell+1}+f(x_1)+x_{\ell},\\ &x_{\ell+2}+x_{\ell+1},\dots,x_{k}+x_{k-1},x_{k+1}+f(x_1)+\beta\Big)
  \end{align*}
 is strictly ergodic and satisfies the following properties:

  \begin{enumerate}
   \item $Z_k$ is a nontrivial extension of $Z_{k-1}$
      \item $Z_i$ is not measurably isomorphic to $X_i$ for all $i\geq 1$
      \item $X_i$ is topologically isomorphic to $X_{i+1}$ for all $i \geq \ell$
  \end{enumerate}
where $Z_k$ is the maximal measurable nilfactor of $(\mathbb{T}^{k+1},m_{\mathbb{T}^{k+1}},T)$ and $X_\ell$ is the maximal topological nilfactor of $(\mathbb{T}^{k+1},m_{\mathbb{T}^{k+1}},T)$.
\end{proposition}

\begin{proof}
  By Remark \ref{two coboundaries does not change much}, there exists a measurable but not continuous coboundary $f\colon\mathbb{T} \rightarrow \mathbb{T}$ over $(\mathbb{T},m_{\mathbb{T}},T_{\alpha})$ such that $(\mathbb{T}^{k+1},m_{\mathbb{T}^k},T)$ is strictly ergodic and measurably isomorphic to the system below  $$(x_1,x_2,\dots,x_{k+1})\mapsto(x_1+\alpha,x_2+x_1,\dots,x_{k}+x_{k-1},x_{k+1}+\beta).$$ Therefore, for all $1 \leq i \leq k$, the $Z_i$ factor of $(\mathbb{T}^{k+1},m_{\mathbb{T}^{k+1}},T)$ is given by $$(x_1,x_2,\dots,x_{i+1})\longmapsto (x_1+\alpha,x_2+\beta,x_3+x_1,x_4+x_3,\dots,x_{i+1}+x_{i})$$ (see \cite{HK2}, Chapter 11, Proposition 1). Property (1) clearly follows.
  To verify property (2), it suffices to show $Z_1$ is not measurably isomorphic to $X_1$. Indeed, writing $f=F\circ T_{\alpha}-F$ for some measurable but not continuous $F\colon \mathbb{T} \rightarrow \mathbb{T}$, we see $g(x_1,x_2,\dots,x_{k+1})=e(x_1)e(x_2-F(x_1))$ is a measurable but not continuous eigenfunction of $(\mathbb{T}^{k+1},m_{\mathbb{T}^{k+1}},T)$. Finally, we verify property (3) by showing $X_{\ell+1}$ is an $l$-step nilsystem.

  Notice $X_{\ell+1}$ lies as an intermediate factor between $Z_{\ell+1}$, whose transformation is defined by $$(x_1,x_2,\dots,x_{\ell+2}) \longmapsto (x_1+\alpha,x_2+\beta,x_3+x_1,x_4+x_3,\dots,x_{\ell+2}+x_{\ell+1}),$$ and the system given by $$(x_1,x_3,\dots,x_{\ell+1})\longmapsto(x_1+\alpha,x_3+x_1,x_4+x_3,\dots,x_{\ell+1}+x_{\ell}).$$ Moreover, the extension from the latter system to $Z_{\ell+1}$ is a group extension by $\mathbb{T}^{2}$. 
  Applying Lemma \ref{intermediate extensions}, we deduce 
  $X_{\ell+1}$ is the quotient of $Z_{\ell+1}$ by the action of some closed subgroup $H \leq \mathbb{T}^{2}$, acting by vertical translations. Since this action is invariant under the transformation of $Z_{\ell+1}$, we notice $H$ must then be invariant under the matrix \[
A=\begin{pmatrix}
1 & 0 \\
1 & 1
\end{pmatrix}.
\]
All such nontrivial subgroups are either $\mathbb{T}^{2}$, $\mathbb{Z}_n \times \mathbb{T}$, or $\mathbb{Z}_n \times \mathbb{Z}_m$ for some $n,m\in \mathbb{N}$ with $n|m.$ If $H=\mathbb{T}^{2}$, then we are done. If $H=\mathbb{Z}_n \times \mathbb{T}$, then the intermediate factor that gives rise is $$(x_1,x_2,\dots,x_{\ell+2}) \mapsto (x_1+\alpha,x_3+x_1,\dots,x_{\ell+1}+x_{\ell},x_2+n\beta,x_{\ell+2}+x_{\ell+1})$$ and if $H=\mathbb{Z}_n \times \mathbb{Z}_m$ we obtain $$(x_1,x_2,\dots,x_{\ell+2}) \mapsto(x_1+\alpha,x_3+x_1,\dots,x_{\ell+1}+x_{\ell},x_2+n\beta,x_{\ell+2}+mx_{\ell+1}).$$ Continuing as in the proof of property (3) in Proposition \ref{proposition with one coboundary on k}, there is no continuous factor map down to the intermediate factors aforementioned. Hence, $X_{\ell+1}$ is given by $$(x_1,x_3,\dots,x_{\ell+1})\longmapsto(x_1+\alpha,x_3+x_1,x_4+x_3,\dots,x_{\ell+1}+x_{\ell}),$$ an $\ell$-step nilsystem. \end{proof}

\noindent\textbf{Remark.}
    In the examples in Proposition \ref{proposition with one coboundary on k} and \ref{proposition with two coboundaries on k} where $Z_1$ is not measurably isomorphic to $X_1$, it is necessary to have at least a two-dimensional $Z_1$ factor. This is because every minimal homeomorphism on $\mathbb{T}$ is topologically isomorphic to an irrational rotation (follows from \cite{Poincare}).

\medskip

  The next lemma is the only place where we use the regionally proximal relation. Given a minimal system $(X,T)$ and $k\in \mathbb{N}$, two points $x,y\in X$ are \textit{regionally proximal of order $k$} if for all $\delta >0$ there exist $x',y' \in X$ and a vector $\textbf{n}=(n_1,n_2,\dots,n_k) \in \mathbb{Z}^{k}$ such that $d(x,x')<\delta, d(y,y')<\delta$, and $$d(T^{n \cdot \epsilon} x', T^{n\cdot \epsilon} y')<\delta \;\mathrm{for \; any}\; \epsilon \in \{0,1\}^{k}, \epsilon\neq(0,0,\dots,0),$$ where $n \cdot \epsilon=\sum_{i=1}^{k} \epsilon_i n_i.$ The set of regionally proximal points of order $k$ is denoted by $RP^{[k]}(X)$, and is called the \textit{ regionally proximal relation of order $k$}. It is a closed $T$-invariant equivalence relation on $X$ and the quotient of $X$ by this relation gives \textit{the maximal topological $k$-step nilfactor} of $(X,T).$ The proof of the latter can be found in \cite{Host Kra and Maass} under the extra assumption $(X,T)$ is distal or in \cite{Shao and Ye} where this assumption was later removed.

\begin{lemma}\label{product of top nil is top nil}
     Let $(X,T)$, $(Y,S)$, and $(X\times Y, T\times S)$ be minimal systems. Then $$X_k(X\times Y,T\times S)= X_k(X,T) \times X_k(Y,S)$$ for all $k \in \mathbb{N}$.
\end{lemma}

\begin{proof}
    We equivalently show $RP^{[k]}(X\times Y) = RP^{[k]}(X) \times RP^{[k]}(Y)$. By the universal property of $X_k(X \times Y)$, we obtain $RP^{[k]}(X\times Y) \subseteq RP^{[k]}(X) \times RP^{[k]}(Y)$. For the reverse containment, let $p,p'\in X$ and $q,q'\in Y$ be regionally proximal points of order $k$ and let $\delta_1,\delta_2>0$. Then there exist $p_1,p_1'$, $q_1,q_1'$, and vectors $\textbf{n}, \textbf{n}' \in \mathbb{Z}^{k}$ such that $d(p,p_1)<\delta_1$, $d(p',p_1')<\delta_1$, $d(q,q_1)<\delta_2$, $d(q',q_1')<\delta_2$ and $d(T^{\textbf{n}\cdot \epsilon}p_1,T^{\textbf{n}\cdot \epsilon}p_1')<\delta_1$ and $d(S^{\textbf{n}'\cdot \epsilon}q_1, S^{\textbf{n}'\cdot \epsilon}q_1')<\delta_2.$ It follows by definition of $RP^{[k]}$ that $(p,q)$ and $(p',q')$ are regionally proximal of order $k$.
\end{proof}

We apply Lemma \ref{product of top nil is top nil} to determine the $X_k$ factor of a product of two systems on $\mathbb{T}^{k}$ that are each measurably but not topologically isomorphic to $(x_1,x_2,\dots,x_k)\mapsto(x_1+\alpha,x_2+x_1,\dots,x_k+x_{k-1})$ for some $\alpha \notin \mathbb{Q}$. The measure-theoretic analog of Lemma \ref{product of top nil is top nil} is given in \cite{Radic} and for convenience we use it in the following proof.


\begin{proposition} \label{the combined proposition}
    Let $0< j < \ell \leq k$, ${\alpha}=(\alpha_1,\alpha_2) \in \mathbb{T}^{2}$ be such that $\alpha_1,\alpha_2 \in \mathbb{T} \setminus\mathbb{Q}$ are rationally independent. Then there exist measurable but not continuous coboundaries $f_1,f_2\colon\mathbb{T}^{2}\rightarrow \mathbb{T}^{2}$ over $(\mathbb{T}^2,m_{\mathbb{T}^2},T_{\alpha})$ such that the system $(\mathbb{T}^{2k},m_{\mathbb{T}^{2k}},T)$ given by 
     \begin{align*}
T\Big(\overline{y}_1,\overline{y}_2,\dots, \overline{y}_k\Big)=\Big(&\overline{y}_1+{\alpha},\overline{y}_2+\overline{y}_1,\dots,\overline{y
}_{j}+\overline{y
}_{j-1},\overline{y}_{j+1}+f_1(\overline{y}_1)+\overline{y}_{j},\overline{y}_{j+2}+\overline{y}_{j+1},\\&\dots,\overline{y}_{\ell}+\overline{y}_{\ell-1},\overline{y}_{\ell+1}+f_2(\overline{y}_1)+\overline{y}_{\ell},\overline{y}_{\ell+2}+\overline{y}_{\ell+1}\dots, \overline{y}_k+\overline{y}_{k-1}\Big)
     \end{align*}
 is strictly ergodic and satisfies the following properties:

  \begin{enumerate}
   \item $Z_k$ is a nontrivial extension of $Z_{k-1}$
      \item $Z_{i}$ and $X_{i}$ are topologically isomorphic, for all $0 \leq i \leq j$
      \item $Z_i$ is not measurably isomorphic to $X_i$ for all $j+1 \leq i \leq k$
      \item $X_i$ is topologically isomorphic to $X_{i+1}$ for all $i \geq \ell$
  \end{enumerate}
where $Z_k$ is the maximal measurable nilfactor and $X_\ell$ is the maximal topological nilfactor of $(\mathbb{T}^{2k},m_{\mathbb{T}^{2k}},T)$ and $X_\ell$ is the maximal topological nilfactor of $(\mathbb{T}^{2k},m_{\mathbb{T}^{2k}},T)$.
\end{proposition}

\begin{proof}

Let $T_1(x_1,x_3,\dots,x_{2k-1})=(x_1+\alpha_1,x_3+x_1,\dots,x_{2k-1}+x_{2k-3})$ and \newline $T_2(x_2,x_4,\dots,x_{2k})=(x_2+\alpha_2,x_4+x_2,\dots,x_{2k}+x_{2k-2})$ where $T_1,T_2\colon\mathbb{T}^{k}\rightarrow \mathbb{T}^{k}.$ By Lemma \ref{coboundary is a coboundary is a ...}, there exist measurable but not continuous coboundaries $h,g\colon\mathbb{T} \rightarrow \mathbb{T}$ over $(\mathbb{T},m_{\mathbb{T}},T_{\alpha_1})$ and $(\mathbb{T},m_{\mathbb{T}},T_{\alpha_2})$ respectively, such that the system on $\mathbb{T}^{k}$ given by 
\begin{align*}
T_1'(x_1,x_3,\dots,x_{2k-1})=(&x_1+\alpha_1,x_3+x_1,\dots,x_{2\ell-1}+x_{2\ell-3},x_{2\ell+1}+h(x_1)+x_{2\ell-1},\\ &x_{2\ell+3}+x_{2\ell+1},\dots,x_{2k-1}+x_{2k-3})
\end{align*}
is measurably isomorphic to $(\mathbb{T}^k,m_{\mathbb{T}^k},T_1)$, and the system on $\mathbb{T}^{k}$ given by 
\begin{align*}
T_2'(x_2,x_4,\dots,x_{2k})= (&x_2+\alpha_2,x_4+x_2,\dots,x_{2j}+x_{2j-2},x_{2j+2}+g(x_1)+x_{2j},\\ &x_{2j+4}+x_{2j+2},\dots,x_{2k}+x_{2k-2})
\end{align*}
 is measurably isomorphic to $(\mathbb{T}^k,m_{\mathbb{T}^k},T_2)$. Letting $f_1(\overline{y}_1)\coloneqq(0,g(x_2))$, $f_2(\overline{y}_1)\coloneqq(h(x_1),0)$, and $T$ be defined as in the statement of the proposition, notice $$T=T_1' \times T_2'.$$ Note in the proof of Proposition \ref{proposition with one coboundary on k} we have already computed the measurable and topological nilfactors of $T_1'$ and $T_2'$. Properties (1)-(4) then all follow by Lemma \ref{product of top nil is top nil} and Proposition 9.1 in \cite{Radic}. \end{proof}
 
 \subsection{Strong orbit equivalence}\label{strong orbit equivalence}

Given two systems $(Y,S)$ and $(Z,T)$, a homeomorphism $h\colon Y \rightarrow Z$ is a \textit{strong orbit equivalence} if $(T\circ h)(y)=(h \circ S^{j(y)})(y)$ and $(h \circ S)(y)=(T^{m(y)} \circ h)(y)$ for some $j,m\colon Y \rightarrow \mathbb{Z}$ with each having at most one point of discontinuity. Let $\alpha \notin \mathbb{Q}$ and $(X_\alpha,\sigma)$ denote the Sturmian subshift. By Corollary 23 of \cite{nontrivial Kronecker trivial MEF}, there exists a minimal topologically weak mixing system $(X,T)$ which is strong orbit equivalent to $(X_{\alpha},\sigma)$. It follows from Theorem 2.2 in \cite{Giordano}, that if $(Y,S)$ is uniquely ergodic and orbit equivalent to $(Y',S')$ (in particular, strong orbit equivalent), then $(Y',S')$ is uniquely ergodic. Hence, $(X,T)$ is uniquely ergodic. We use this example in Proposition \ref{nontrivial measurable nils but trivial topological nils}.

\subsection{A system with nontrivial measurable nilfactors but trivial topological nilfactors} To construct such an example, we use the following lemma.


\begin{lemma}\label{measurable but not continuous coboundary over minimal cantor system}
    Let $(X,\mu,T)$ and $(Y,\nu,S)$ be strong orbit equivalent and $f\colon X \rightarrow \mathbb{T}$ be a measurable but not continuous coboundary over $(X,\mu,T)$ whose solution has more than one discontinuity. Then there exists a measurable but not continuous coboundary over $(Y,\nu,S)$ taking values in $\mathbb{T}.$
\end{lemma}

\begin{proof}
Let $h\colon Y \rightarrow X$ be a homeomorphism such that $(T\circ h)(y)=(h \circ S^{j(y)})(y)$ and $(h \circ S)(y)=(T^{m(y)} \circ h)(y)$ for some $j,m\colon Y \rightarrow \mathbb{Z}$ with each having at most one point of discontinuity. From the preceding equality we have 
\begin{equation}\label{eq: orbit equivalence}
S (h^{-1}(x))= h^{-1}(T ^{m \circ h^{-1}(x)}(x)).
\end{equation}

 Let $f=F\circ T-F$ for some measurable but not continuous $F\colon X \rightarrow \mathbb{T}.$ Then we have
\begin{equation}\label{eq:f h is a coboundary over S}
\begin{aligned}
f\circ h (y)&= F\circ T\circ h(y)-F \circ h(y)= F \circ h \circ S^{j(y)}(y)- F \circ h(y)\\ &= G \circ S(y)-G(y)
\end{aligned}
\end{equation}
where $$G(y):= \sum_{k=0}^{j(y)-1} (F \circ h \circ S^{k})(y).$$ We show $f \circ h$ is our desired measurable but not continuous coboundary over $(Y,\nu,S)$. First, it follows from \eqref{eq: orbit equivalence} and \eqref{eq:f h is a coboundary over S},
\begin{align*}
f(x)&= G \circ S \circ h^{-1} (x)-G \circ h^{-1} (x)= G \circ h^{-1} \circ T^{m \circ h^{-1}(x)}(x) - G \circ h^{-1} (x)\\ &= F' \circ T (x) - F'(x)
\end{align*}
where $${F'}(x)=\sum_{k=0}^{m(h^{-1}(x))-1} (G \circ h^{-1} \circ T^{k})(x).$$ Recall $F$ is not continuous, therefore, ${F'}$ is not continuous by Remark \ref{cob soln are unique}. Moreover, it has more than one discontinuity. Notice if $m$ is continuous, then we are done. Otherwise, let $z\in X$ be a discontinuity point of $m$ and let $x' \in X$ be a point of discontinuity for ${F'}$ such that $ h^{-1}(x') \neq z$. Then if $G$ were continuous, ${F'}$ would also be continuous, giving a contradiction. \end{proof}

\begin{proposition} \label{nontrivial measurable nils but trivial topological nils}
    For each $k\in \mathbb{N}$, there exists a strictly ergodic topologically weak mixing system such that $Z_k$ is a nontrivial extension of $Z_{k-1}$.
\end{proposition}

\begin{proof}
    For $k=1$, we take the example $(X,T)$ as defined in Section \ref{strong orbit equivalence} and for $k \geq 2$, we build an extension of $X$. Let $f\colon \mathbb{T}\rightarrow \mathbb{T}$ be the measurable but not continuous coboundary over $(\mathbb{T},m_{\mathbb{T}},T_{\alpha})$ constructed in Lemma \ref{coboundary is a coboundary is a ...}. Write $f=F\circ T_{\alpha}-F$, where $F\colon \mathbb{T} \rightarrow \mathbb{T}$ is not continuous. Let $\pi\colon X_{\alpha} \rightarrow \mathbb{T}$ be the measurable isomorphism between $(X_\alpha,\sigma)$ and $(\mathbb{T},T_{\alpha})$. 
Note $\pi^{-1}$ is not continuous, however, $\pi$ is continuous. Then $$F\circ \pi\circ \sigma -F\circ \pi$$ is a measurable but not continuous coboundary over the system $(X_{\alpha},\sigma)$. Since $T$ and $\sigma$ are strong orbit equivalent, by Lemma \ref{measurable but not continuous coboundary over minimal cantor system}, there exists a measurable but not continuous coboundary $h\colon X \rightarrow \mathbb{T}$ over $(X,\mu,T)$ that we write as $$h=G \circ T-G.$$ Consider the following extension of $X$ given by 
    $$S(x_1,x_2,\dots,x_{k+1})=(Tx
    _1, x_2+h(x_1)+\beta, x_3+h(x_1)+x_2,\dots,x_{k+1}+h(x_1)+x_k),$$ where $S\colon X \times \mathbb{T}^{k} \rightarrow X \times \mathbb{T}^{k}$ and $\beta \notin \mathbb{Q}$ rationally independent from $\alpha$.
     Then $(X \times \mathbb{T}^{k},\mu\times m_{\mathbb{T}^{k}},S)$ is measurably isomorphic to $$S'(x_1,x_2,\dots,x_{k+1})=(Tx_1,x_2+\beta,x_3+x_2,\dots,x_{k+1}+x_k);$$ the isomorphism being given by the map $$(x_1,x_2,\dots,x_{k+1}) \mapsto (x_1,x_2-G(x_1),\dots, x_{k+1}-G(x_1)).$$ Since ergodicity transfers through orbit equivalences, $S$ is ergodic. Strict ergodicity follows by $S$ being a sequence of ergodic group extensions of uniquely ergodic systems and $\mathrm{Supp}(\mu \times m_{\mathbb{T}^{k}})=X\times  \mathbb{T}^{k}$ (see Proposition 3.10 of \cite{Fur81}).
     
     The eigenfunctions of $(X \times \mathbb{T}^{k},\mu\times m_{\mathbb{T}^{k}},S)$ are then of the form $$f_{n,m}(x_1,x_2,\dots,x_{k+1}):=f_{n}(x_1)e(mx_2)e(-mG(x_1)), \;\;n,m \in \mathbb{Z},$$ where $f_{n}\colon X \rightarrow \mathbb{T}$ are the eigenfunctions of $T$. Note $mG$ is not continuous because the measurable not continuous coboundary constructed in Lemma \ref{coboundary is a coboundary is a ...}, $mF$, $m\neq 0$, is discontinuous on the set $\{n\alpha\colon n \in \mathbb{N}\}$ and then we can apply Lemma \ref{measurable but not continuous coboundary over minimal cantor system} to $mf$. This implies there are no non-constant continuous eigenfunctions of $S$. Hence $X_1$ is trivial. Furthermore, the $k$-step nilsystem $$(x_1,x_2,\dots,x_k)\mapsto(x_1+\beta,x_2+x_1,\dots,x_k+x_{k-1})$$ is a clear factor of $(X \times \mathbb{T}^{k},\mu\times m_{\mathbb{T}^{k}},S)$, concluding the proof. \end{proof}

\subsection{Proof of Theorem \ref{the theorem in introduction}}

  An example realizing the case when $j=\ell=0$ is in Proposition \ref{nontrivial measurable nils but trivial topological nils}. The remaining cases are realized by Propositions \ref{proposition with one coboundary on k}, \ref{proposition with two coboundaries on k}, and \ref{the combined proposition}.

\subsection{Flexibility for $\mathbb{Z}^{d}$ actions}

In \cite{Frantzikinakis and Kra}, it is shown for a system with $d$ commuting transformations $(X,\mu,T_1,T_2\dots,T_d)$, $T_1$,...,$T_d$, share the same sequence of measurable nilfactors. The topological analog was proven in \cite{Shao and Xu}. This allows us to generalize the examples in Proposition \ref{proposition with one coboundary on k}, \ref{proposition with two coboundaries on k}, \ref{the combined proposition} to $\mathbb{Z}^{d}$ actions. For any $c\in \mathbb{T}$, let $T_c\colon \mathbb{T}^{k}\rightarrow \mathbb{T}^{k}$ be
\begin{align*}
    T_c(x_1,x_2,\dots,x_{k})=\Big(&x_1+\alpha,x_2+x_1,\dots,x_{j}+x_{j-1},x_{j+1}+f(x_1)+x_{j}+c,\\
    &x_{j+2}+x_{j+1},\dots,x_{k}+x_{k-1}\Big).
  \end{align*} where $f$ is the measurable but not continuous coboundary constructed in Lemma \ref{coboundary is a coboundary is a ...}.
Then for any constants $c_i\in \mathbb{T}$, $(\mathbb{T}^{k},m_{\mathbb{T}^{k}},T_{c_1},\dots,T_{c_d})$ is strictly ergodic because $T_{c_1}$ is strictly ergodic and $T_{c_i}$ all preserve Haar measure. Proceeding as in the proof of Proposition \ref{proposition with one coboundary on k}, for all $i$, properties (1)-(4) hold for $T_{c_i}$. We generalize Propositions \ref{proposition with two coboundaries on k} and \ref{the combined proposition} similarly by taking transformations which add a constant in the same coordinate as the coboundary. The $\mathbb{Z}^{d}$ analog of the example with nontrivial measurable nilfactors and trivial topological nilfactors exists by a generalization of Lehrer's result \cite{Lehrer} to $\mathbb{Z}^{d}$ actions given by Rosenthal \cite{Rosenthal}.

\medskip

\noindent\textbf{Acknowledgments.} 
The author is grateful to Bryna Kra for her continued guidance, helpful discussions, and feedback throughout the project. The author would also like to thank Amadeus Maldonado and Trist\'an Radi\'c for many enlightening discussions and for comments on a final draft. Finally, the author thanks Keith Burns, Anh Le, Redmond McNamara, and Wenbo Sun for fruitful discussions.

 \end{proof}
 
\end{comment}

\end{document}